\documentclass[11pt]{article}
\usepackage[margin=1in]{geometry}
\usepackage{amsmath,amssymb,amsthm,mathtools,amsfonts}
\usepackage[colorlinks=true,linkcolor=blue,citecolor=blue,urlcolor=blue]{hyperref}
\usepackage{authblk}
\usepackage[sort]{cite}
\usepackage{tikz}
\usetikzlibrary{backgrounds,fit,positioning}
\usepackage{graphicx}
\usepackage{tikz}
\usetikzlibrary{backgrounds,calc}

\allowdisplaybreaks
\numberwithin{equation}{section}

\newtheorem{theorem}{Theorem}
\newtheorem{lemma}[theorem]{Lemma}
\newtheorem{corollary}[theorem]{Corollary}

\newtheorem{problem}[theorem]{Problem}
\newtheorem{proposition}[theorem]{Proposition}

\newtheorem{clm}{Claim}

\newcommand \Clm[2]
{
\begin{clm}\label{#1}
#2
\end{clm}
}

\def \proof {\noindent {\it Proof}. }

\newcounter{countclaim}

\newcommand{\claimend}{{\hfill$\natural $}}

\newcommand \aln[2]
{
\begin{align}\label{#1}
#2
\end{align}
}

\title{Real-rooted flow polynomials have only integral roots}

\author[1]{Meiqiao Zhang\thanks{Corresponding author. Email: 
		meiqiaozhang95@163.com and meiqiaozhang@xmu.edu.cn.}}
\author[2]{Fengming Dong\thanks{Email: fengming.dong@nie.edu.sg %(expired on 24/03/2027) 
		and donggraph@163.com.}}

\affil[1]{\small School of Mathematical Sciences, Xiamen  University, China}

\affil[2]{\small
	National Institute of Education,
	Nanyang Technological University, 
	Singapore}

\date{}

\begin{document}

\maketitle

\begin{abstract}
Let $G$ be a connected bridgeless graph. In 2011, Kung and Royle showed that all roots of the flow polynomial $F(G,\lambda)$ of $G$ are integers if and only if $G$ is the dual of a chordal plane graph. In this article, we further prove that if $F(G,\lambda)$ has real roots only, then $G$ is the dual of a chordal plane graph and each root of $F(G,\lambda)$ is an integer in the set $\{1,2,3\}$.
%Let $G$ be a bridgeless graph and let $F(G,\lambda)$ denote the flow polynomial of $G$.In 2011, Kung and Royle showed that all roots of $F(G,\lambda)$ are integers if and only if $G$ is the dual of a chordal plane graph. In this paper, we prove that $F(G,\lambda)$ has only real roots if and only if $F(G,\lambda)$ has only integer roots.Consequently, all roots of $F(G,\lambda)$ are real if and only if $G$ is the dual of a chordal plane graph. Further,by the duality between chromatic polynomials and flow polynomials, it follows immediately that a planar graph has only real chromatic roots if and only if all of its chromatic roots are integers, and every planar graph with only real chromatic roots is chordal.
\end{abstract}



\section{Introduction}
All graphs considered in this paper are finite and undirected, and may have loops and parallel edges. For any graph $G$, let $V(G)$ and $E(G)$ be the vertex set and edge set of $G$, respectively. 

In 1912, Birkhoff introduced the chromatic polynomial as a tool to attack the Four-Color Conjecture~\cite{Birk1912} . For any graph $G$, the \textit{chromatic polynomial} $P(G,\lambda)$ of $G$ counts the number of proper $\lambda$-colorings of $G$ for each positive integer $\lambda$. Thus, the Four-Color Conjecture is equivalent to the assertion that $P(G,4)>0$ for all loopless planar graphs $G$.
It is then well known, from the deletion-contraction formula, that this counting function is in fact a polynomial in $\lambda$ with integer coefficients, hence the name. 
%More precisely, $P(G,\lambda)$ is a monic polynomial with integer coefficients and degree $|V(G)|$. Therefore, it can naturally be viewed and studied as a polynomial over the real or complex numbers. 
Over the years, the chromatic polynomial has become one of the fundamental graph polynomials and an important object of study, due to the close connections between its algebraic properties and the structural properties of the underlying graph. See, e.g.,~\cite{Dong2021, Dong2005, Read1988, Royle2009} for background.

In particular, the roots of $P(G,\lambda)$ are called the \textit{chromatic roots} of $G$. Clearly, $0$ is a chromatic root of every loopless graph. 
%Moreover, the location and distribution of chromatic roots have been extensively studied, both on the real line and in the complex plane. For general graphs, there are no real chromatic roots in $(-\infty,0)$, $(0,1)$, or $\left(1,\frac{32}{27}\right]$~\cite{Jackson}, and these are precisely the maximal intervals on the real line that are free of chromatic roots~\cite{Thomassen}. In contrast, chromatic roots are dense in the entire complex plane. Even when restricted to planar graphs, chromatic roots are dense in the complex plane with the possible exception of the disc $|z-1|<1$~\cite{Sokal}.
Also, it can be easily verified that all chromatic roots of a chordal graph are nonnegative integers, where a graph is \textit{chordal} if it is simple and contains no induced cycle of length greater than three.
%Moreover, chordal graphs form a particularly interesting class in this context. A graph is \textit{chordal} if it contains no induced cycles of length greater than three. Due to the perfect elimination ordering of any chordal graph, it is easy to see that all of its chromatic roots are nonnegative integers. 
Meanwhile, there exist non-chordal graphs whose chromatic roots are all integers~\cite{Read,Dong2002,Dong2005,Dong1998,Dmitriev}, as well as non-planar graphs whose chromatic roots are all real but are not all integers~\cite{Dong2018}. 
Motivated by these observations on integral and real chromatic roots, it is natural to consider the following question, which remains widely open.

\begin{problem}[\cite{Dong2018sur}]\label{prob1}
Is there a planar graph that has real chromatic roots only and contains non-integral chromatic roots?
\end{problem}

Clearly, for Problem~\ref{prob1}, it suffices to consider non-chordal planar graphs.

On the other hand, the flow polynomial was introduced by Tutte~\cite{Tutte1947} in 1950 as a natural counterpart of the chromatic polynomial. For any graph $G$, the \textit{flow polynomial} $F(G,\lambda)$ of $G$ counts the number of nowhere-zero $\Gamma$-flows on any orientation $D$ of $G$, where $\Gamma$ is an arbitrary additive abelian group of order $\lambda$, for each positive integer $\lambda$. It is well known that $F(G,\lambda)$ is also a polynomial in $\lambda$ with integer coefficients and is independent of the choice of $\Gamma$, as can be seen directly from the following deletion-contraction formula~\cite{Tutte1984}:
$$
F(G,\lambda)=
\left\{
\begin{array}{@{}ll}
1, 
& \text{if } E(G)=\varnothing,\\
0, 
& \text{if } G \text{ has a bridge},\\
F(G_1,\lambda)F(G_2,\lambda), 
& \text{if } G=G_1\cup G_2,\\
(\lambda-1)F(G-e,\lambda), 
& \text{if } e \text{ is a loop},\\
F(G/e,\lambda)-F(G-e,\lambda), 
\qquad\quad& \text{if } e \text{ is neither a loop nor a bridge},
\end{array}
\right.
$$
where $G/e$ and $G-e$ are the graphs obtained from $G$ by contracting $e$ and deleting $e$ respectively, and $G_1\cup G_2$ is the disjoint union of graphs $G_1$ and $G_2$.

In analogy with chromatic roots,  the roots of $F(G,\lambda)$ are called the \textit{flow roots} of $G$. 
For any connected plane graph $G$, a classical duality due to Tutte~\cite{Tutte1947} shows that 
\aln{dual}
{P(G,\lambda)=\lambda F(G^*,\lambda),}
where $G^*$ is the dual plane graph of $G$. 
Thus, Problem~\ref{prob1} can be equivalently stated in terms of flow roots as follows.
   
\begin{problem}[\cite{Dong2018,Dong2018sur}]\label{prob2}
Is there a planar graph that has real flow roots only and contains non-integral flow roots?
\end{problem}

Problem~\ref{prob2} also remains open. In fact, even the following weaker question is still unsolved.

\begin{problem}[\cite{Dong2018,Dong2018sur}]\label{prob3}
Is there a graph  that has real flow roots only and contains non-integral flow roots?
\end{problem}

When considering Problems~\ref{prob2} and~\ref{prob3}, we restrict our attention to connected bridgeless graphs $G$, as $F(G,\lambda)$ factorizes over the components of $G$ if $G$ is disconnected, and $F(G,\lambda)$ is identically zero if $G$ contains a bridge.
Before introducing our main results, we first review some relevant developments.

In 2011, Kung and Royle characterized the graphs whose flow roots are all integers~\cite{KungRoyle2011}.

\begin{theorem}[\cite{KungRoyle2011}]\label{kung}
If $G$ is a connected bridgeless graph, then its flow roots are integral if and only if $G$ is the dual of a chordal plane graph.
\end{theorem}

Inspired by Theorem~\ref{kung}, it is natural to ask whether every graph whose flow roots are all real must also be the dual of a chordal plane graph~\cite{Dong2018sur}.
%By Theorem~\ref{kung}, it remains to determine for Problems~\ref{prob2} and~\ref{prob3} whether there exists a planar graph or a graph that is not the dual of any chordal plane graph but has only real flow roots.
Moreover, Dong extended Theorem~\ref{kung} by determining all possible flow roots of graphs with only integral flow roots as follows~\cite{Dong2018}.

\iffalse
\begin{theorem}[\cite{Dong2018}]\label{thm:dong}
Let $G$ be a graph with real flow roots only. If some flow roots of $G$ are not in the set $\{1,2,3\}$, then $|V|+17\le |E|< (32|V|-49)/5$ and $G$ has at least 9 flow roots in the interval $(1,2)$.
\end{theorem}
\fi

\begin{theorem}[\cite{Dong2018}]\label{cor:dong}
Let $G$ be a connected bridgeless graph with real flow roots only. The following statements are equivalent:
\begin{enumerate}
\item $G$ is the dual of some chordal plane graph;
\item each flow root of $G$ is in the set $\{1,2,3\}$;
\item $G$ has no flow roots in the interval $(1,2)$.
\end{enumerate}
\end{theorem}

In this paper, building on the aforementioned results, we settle Problems~\ref{prob1},~\ref{prob2}, and~\ref{prob3} in the negative.

\begin{theorem}\label{thm:main}
Let $G$ be a connected bridgeless graph.  Then $G$ has only real flow roots if and only if $G$ has only integral flow roots if and only if
$G$ is the dual of a chordal plane graph. Moreover, every flow root of such a graph belongs to $\{1,2,3\}$.
\end{theorem}

The next corollary then follows from (\ref{dual}).

%\begin{corollary}\label{coro1}
%If $G$ is a bridgeless graph, then its flow roots are real if and only if $G$ is the dual of a chordal plane graph.  In particular, all of its flow roots belong to $\{1,2,3\}$.
%\end{corollary}

\begin{corollary}\label{coro2}
Let $G$ be a loopless planar graph.  Then $G$ has only real chromatic roots if and only if $G$ has only integral chromatic roots if and only if $G$ is chordal.
Moreover, every chromatic root of such a graph belongs to $\{0,1,2,3\}$.
\end{corollary}

We shall present some preliminary results in Section~\ref{sec2} and prove Theorem~\ref{thm:main} in Section~\ref{sec3}.

\section{Preliminaries
\label{sec2}}
In this section, we introduce some notation and  properties of flow polynomials that will be used in the proof of Theorem~\ref{thm:main}.

Let $G$ be a graph. 
%For any proper subset $V_0$ of $V(G)$,let $G-V_0$ denote the graph obtained from $G$ by deleting all the vertices in $V_0$ and all the edges incident with some vertex in $V_0$.  
%Moreover, we simply write $G-v$ for $G-\{v\}$ and  $G-e$  for $G-\{e\}$. 
%For any two nonadjacent vertices $u$ and $v$ in $G$, let $G+uv$ be the graph obtained from $G$ by adding an edge between $u$ and $v$.
%For any $v\in V(G)$, $v$ is called a \textit{cut-vertex} if $G-v$ has more components than $G$ has.
$G$ is said to be \textit{nonseparable} if either $|E(G)|\le |V(G)|=1$ or $G$ is connected without loops and cut-vertices, and \textit{separable} otherwise. A \textit{block} of $G$ is a maximal nonseparable subgraph of $G$.
For any $v\in V(G)$, let $d_G(v)$ denote the degree of $v$ in $G$.
For any $S\subseteq E(G)$, let $G-S$ denote the spanning subgraph of $G$ with edge set $E(G)\setminus S$.  
$S$ is called an \textit{edge-cut} of $G$ if $G-S$ has more  components than $G$. A \textit{bridge} is a 1-edge-cut.
An edge-cut $S$ of $G$ is said to be \textit{proper} if $G-S$ has no isolated vertices. $G$ is \textit{$k$-edge-connected} if $G$ is connected with $|V(G)|\ge 2$ and every edge-cut of $G$ has size at least $k$.

The following lemma indicates that the flow polynomial $F(G,\lambda)$ of a connected bridgeless graph $G$ can be factorized or simplified whenever $G$ contains a loop, $G$ is separable, $G$  contains a 2-edge-cut, 
%$G-e$ is separable for some $e\in E(G)$, 
or $G$ contains a proper 3-edge-cut.
\begin{lemma}[\cite{Jackson,Tutte1984}]\label{lem1}
Let $G$ be a connected bridgeless graph. 
\begin{enumerate}
\item If $G$ contains a loop $e$, then $F(G,\lambda)=(\lambda-1)F(G-e,\lambda).$
%\item If $G$ is disconnected and $G_1,G_2,\dots, G_k$ are the components of $G$, then $$F(G,\lambda)=\prod_{i=1}^kF(G_i,\lambda).$$
\item If $G_1,G_2,\dots, G_b$ are the blocks of $G$, then 
$$F(G,\lambda)=\prod_{i=1}^bF(G_i,\lambda).$$
\item If $e$ is an edge in a 2-edge-cut of $G$, then $F(G,\lambda)=F(G/e,\lambda)$.
%\item Assume that $G$ is connected, $v\in V(G)$, $e=u_1u_2\in E(G)$. If $G-e$ is separable and $H_1,H_2$ are edge-disjoint subgraphs of $G-e$ such that $E(H_1)\cup E(H_2)=E(G-e)$,  $V(H_1)\cap V(H_2)=\{v\}$, $V(H_1)\cup V(H_2)=V(G)$, $u_1\in V(H_1)$,  and $u_2\in V(H_2)$, then $$F(G,\lambda)=\frac{F(G_1,\lambda)F(G_2,\lambda)}{\lambda-1},$$ where $G_i=H_i+vu_i$ for $i=1,2$. An example is as shown in Figure~\ref{fig:G-minus-e-separable}.
%\item \red{Add the case that $S$ is a 2-edge-cut of $G$. Similar to the case that $S$ is a proper $3$-edge-cut.}
\item Assume that $S$ is a proper 3-edge-cut of $G$ whose removal separates $G$ into two subgraphs  $H_1$ and $H_2$. Let $G_i$ be the graph obtained from $G$ by contracting $E(H_{3-i})$ for $i=1,2$.  Then $$F(G,\lambda)=\frac{F(G_1,\lambda)F(G_2,\lambda)}{(\lambda-1)(\lambda-2)}.$$ 
An example is as shown in Figure~\ref{fig:G-three-edge-cut}.
\end{enumerate}
\end{lemma}

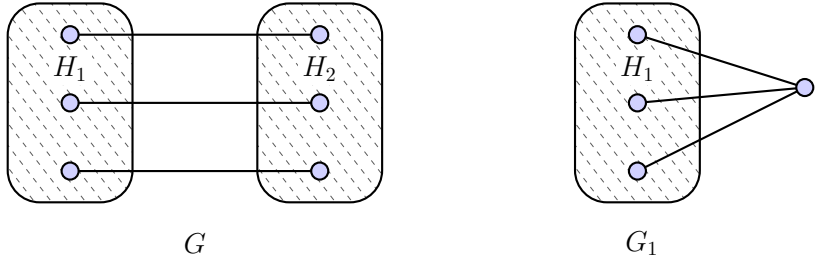
\begin{figure}[htbp]
\centering
\begin{tikzpicture}[
    scale=0.75,
    transform shape,
    x=1cm,
    y=1cm,
    vertex/.style={circle,draw=black,fill=blue!18,thick,
        minimum size=3mm,inner sep=0pt},
    edge/.style={black,thick},
    boundary/.style={black,thick,rounded corners=12pt},
    hatch/.style={black!65,line width=.35pt,
        dash pattern=on 2pt off 2.2pt},
    graphlabel/.style={font=\Large},
    regionlabel/.style={font=\Large,fill=white,inner sep=2pt}
]

%---------------------------------------------------------------
% The graph G.
%---------------------------------------------------------------
\coordinate (a1) at (1.10,2.95);
\coordinate (a2) at (1.10,1.75);
\coordinate (a3) at (1.10,0.55);
\coordinate (b1) at (5.50,2.95);
\coordinate (b2) at (5.50,1.75);
\coordinate (b3) at (5.50,0.55);

% Uniform dashed filling of both rounded regions.
\begin{scope}[on background layer]
    \begin{scope}
        \clip[rounded corners=12pt] (0,0) rectangle (2.20,3.50);
        \foreach \a in {-3.4,-3.1,...,2.3} {
            \draw[hatch] (\a,3.85) -- ++(3.25,-4.20);
        }
    \end{scope}
    \begin{scope}
        \clip[rounded corners=12pt] (4.40,0) rectangle (6.60,3.50);
        \foreach \a in {1.0,1.3,...,6.7} {
            \draw[hatch] (\a,3.85) -- ++(3.25,-4.20);
        }
    \end{scope}

    % The three edges of the edge-cut.
    \draw[edge] (a1)--(b1);
    \draw[edge] (a2)--(b2);
    \draw[edge] (a3)--(b3);
\end{scope}

% Region boundaries.
\draw[boundary] (0,0) rectangle (2.20,3.50);
\draw[boundary] (4.40,0) rectangle (6.60,3.50);

% Vertices.
\foreach \p in {a1,a2,a3,b1,b2,b3} {
    \node[vertex] at (\p) {};
}

% Labels.
\node[regionlabel] at (1.10,2.35) {$H_1$};
\node[regionlabel] at (5.50,2.35) {$H_2$};
\node[graphlabel] at (3.30,-0.72) {$G$};

%---------------------------------------------------------------
% The graph G_1.
%---------------------------------------------------------------
\begin{scope}[xshift=10cm]
    \coordinate (c1) at (1.10,2.95);
    \coordinate (c2) at (1.10,1.75);
    \coordinate (c3) at (1.10,0.55);
    \coordinate (w)  at (4.05,2.02);

    % Uniform dashed filling of H_1.
    \begin{scope}[on background layer]
        \begin{scope}
            \clip[rounded corners=12pt] (0,0) rectangle (2.20,3.50);
            \foreach \a in {-3.4,-3.1,...,2.3} {
                \draw[hatch] (\a,3.85) -- ++(3.25,-4.20);
            }
        \end{scope}

        \draw[edge] (c1)--(w);
        \draw[edge] (c2)--(w);
        \draw[edge] (c3)--(w);
    \end{scope}

    \draw[boundary] (0,0) rectangle (2.20,3.50);

    \foreach \p in {c1,c2,c3,w} {
        \node[vertex] at (\p) {};
    }

    \node[regionlabel] at (1.10,2.35) {$H_1$};
    \node[graphlabel] at (1.18,-0.72) {$G_1$};
\end{scope}

\end{tikzpicture}
\caption{$G$ has a 3-edge-cut}
\label{fig:G-three-edge-cut}
\end{figure}

Moreover, some standard facts of flow polynomials are as follows.

\begin{proposition}[\cite{Wakelin,KungRoyle2011}]\label{prop:standard}
Let $G$ be a bridgeless graph, where $|V(G)|=n$ and $|E(G)|=m$.
\begin{enumerate}
    \item $F(G,\lambda)$ has no real
    roots in $(-\infty,1)$. 

    \item If $m>0$ and $G$ is nonseparable, then $1$ is a flow root of $G$ with multiplicity one.

    \item If $G$ is $3$-edge-connected, then
\aln{}
{
        F(G,\lambda)=\sum_{i=0}^{m-n+1}b_i\lambda^i,
}
where $b_{m-n+1}=1, b_{m-n}=-m$, and $b_{m-n-1}, b_{m-n-2},\dots,b_0$ are all integers.
\end{enumerate}
\end{proposition}

We establish our key lemma below.

\begin{lemma}\label{lem:upper}
Let $G$ be a nonseparable $3$-edge-connected graph with $|V(G)|=n$ and $|E(G)|=m$. 
If  $G$ has only real flow roots, 
then $m\leq 2n-1$.
Moreover, if $m=2n-1$, then
$F(G,\lambda )=(\lambda -1)(\lambda -2)^{n-1}.$
\end{lemma}

\begin{proof}
Since $G$ is a nonseparable $3$-edge-connected graph, 
we have $n\ge 2$ and the degree of every vertex in $G$ is at least three. Then $2m\ge 3n\ge 2n+2$, which implies that $m\ge n+1$.

Let $r=m-n+1$. Then $r\ge 2$. By Proposition~\ref{prop:standard} (iii), we may write
$$F(G,\lambda)=\sum_{i=0}^{r}b_i\lambda^i,$$ 
where $b_r=1$, $b_{r-1}=-m$, and $b_{r-2}, b_{r-3},\dots,b_0$ are all integers.

Since $F(G,\lambda)$ is a polynomial of degree $r\ge 2$, let 
$\rho_1,\rho_2,\ldots,\rho_r$ be the roots of $F(G,\lambda)$, counted with multiplicity. Then by Proposition~\ref{prop:standard} (i) and (ii), we can further assume  that $\rho_1=1$, and $\rho_i>1$ for $2\leq i\leq r$. Clearly, $\sum_{i=1}^{r}\rho_i=-b_{r-1}=m$, implying that
\aln{eq1-1}{
    n-1=m-r=\sum_{i=1}^{r}\rho_i-r=\sum_{i=1}^{r}(\rho_i-1)=\sum_{i=2}^{r}(\rho_i-1).
}

Let $Q(\lambda)
=\prod\limits_{i=2}^{r}(\lambda-\rho_i)$. Then $Q(1)\neq 0$ and $F(G,\lambda)=(\lambda-1)Q(\lambda)$.
Further,
$$
\left\{
\begin{aligned}
&F'(G,\lambda)=\sum_{i=1}^{r}ib_i\lambda^{i-1}; \mbox{ and}\\
&F'(G,\lambda)=Q(\lambda)+(\lambda-1)Q'(\lambda)=\prod_{i=2}^{r}(\lambda-\rho_i)+(\lambda-1)Q'(\lambda).\\
\end{aligned}
\right.
$$
Then it is clear that $
F'(G,1)=\sum\limits_{i=1}^{r}ib_i$ while
$F'(G,1)=Q(1)\neq 0$, which implies that $F'(G,1)$ is a nonzero integer. Moreover, since $\rho_i-1>0$ for $2\le i\le r$, we have $$\prod_{i=2}^{r}(\rho_i-1)=|Q(1)|=|F'(G,1)|
=\left |\sum_{i=1}^{r}ib_i\right |
\ge 1.
$$
Then by (\ref{eq1-1}) and the 
AM-GM inequality,
\aln{ineq1}
{
n-1&=\sum_{i=2}^{r}(\rho_i-1)
\geq (r-1)
        \left(\prod_{i=2}^{r}(\rho_i-1)\right)^{1/(r-1)}
        \geq r-1,
}
where the equality in the first inequality holds if and only if $\rho_i$ are all equal for  $2\le i\le r$.
Hence $r\leq n$, which is equivalent to $m\leq2n-1$.

Moreover, for the case $m=2n-1$, we have $r=n$ and all the equalities in (\ref{ineq1}) hold. Therefore, $\rho_i$ are all equal for $2\le i\le r$. Then by (\ref{eq1-1}), 
$$n-1=\sum_{i=2}^{r}(\rho_i-1)=\sum_{i=2}^{n}(\rho_i-1),$$ 
which implies that $\rho_i=2$ for all $2\le i\le r$.
Hence $F(G,\lambda)=(\lambda-1)(\lambda-2)^{n-1}$.
\qed
\end{proof}

\section{Proof of Theorem~\ref{thm:main}
\label{sec3}}
In this section, we shall use the following result, which appears as Lemma 15 in~\cite{Dong2018} , to prove Theorem~\ref{thm:main}.

\begin{lemma}[\cite{Dong2018}]\label{prop:dong}
Suppose that $G$ is a $3$-edge-connected graph  with $|V(G)|=n$ and $|E(G)|=m$ such that $G$ has only real flow roots. Let $k=\left|\{v\in V(G):d_G(v)>3\}\right|$. If $n\ge 3$ and
$G$ has no proper $3$-edge-cut,
then $n\geq 2k+1$,
and
\begin{equation}\label{eq:dong}
    m\geq 2n+2k-3+\frac{4(k-1)^2}{n-2k}.
\end{equation}
\end{lemma}

\noindent\textit{Proof of Theorem~\ref{thm:main}.}
By Theorems~\ref{kung} and~\ref{cor:dong},
it suffices to show that for any connected bridgeless graph $G$, if $G$ has only real flow roots, then $G$ has only integral flow roots. 

Suppose that $G$ is a counterexample to the above statement with the minimum number of edges. Then $G$ is connected and bridgeless, and $G$ has only real flow roots and contains a non-integral  flow root $x$. Let $n=|V(G)|$, $m=|E(G)|$, and $k=\left|\{v\in V(G):d_G(v)>3\}\right|$. We will derive a contradiction through the following series of claims.

\Clm{cl4}{$G$ is loopless.}

\proof
Assume that $G$ contains a loop $e$. Since $G$ is connected and bridgeless, $G-e$ is also connected and bridgeless. Moreover, by Lemma~\ref{lem1} (i), $$F(G,\lambda)=(\lambda-1)F(G-e,\lambda),$$
which implies that $G-e$ has only real flow roots and contains the non-integral flow root  $x$ with $|E(G-e)|<|E(G)|$, contradicting the minimality of $G$.
Hence $G$ is loopless.
\claimend

\iffalse
\Clm{cl1}{$G$ is connected.}

\proof
Assume that $G$ is disconnected with components $G_1,G_2,\dots,G_c$, where $c\ge 2$ and $|E(G_i)|>0$ for all $1\le i\le c$. Then by Lemma~\ref{lem1} (ii), each $G_i$ has only real flow roots and some $G_i$ has the non-integral flow root  $x$, where each $G_i$ contains fewer edges than $G$,  contradicting the minimality of $G$.
Hence $G$ is connected.
\claimend
\fi

\Clm{cl2}{$G$ is nonseparable with $n\ge 2$.}

\proof
If $n=1$, then $m=0$ due to Claim~\ref{cl4}. As a result, $F(G,\lambda)=1$, contradicting the assumption of $G$. Thus $n\ge 2$. Then by Claim~\ref{cl4}, it remains to show that $G$ contains no cut-vertices.

Assume that $G$ contains cut-vertices. Then $G$ has blocks $G_1,G_2,\dots,G_b$, where $b\ge 2$ and $0<|E(G_i)|<|E(G)|$ for all $1\le i\le b$. 
Since $G$ is bridgeless, each $G_i$ is also  bridgeless.
Then by Lemma~\ref{lem1} (ii), 
$$F(G,\lambda)=\prod_{i=1}^bF(G_i,\lambda),$$
which implies that each $G_i$ has only real flow roots and some $G_i$ has the non-integral flow root  $x$, where each $G_i$ contains fewer edges than $G$, contradicting the minimality of $G$.
Hence Claim~\ref{cl2} holds.
\claimend

\Clm{cl3}{$G$ is $3$-edge-connected.}

\proof
By Claim~\ref{cl2}, it remains to show that $G$ contains no edge-cut of size two.

Assume that $G$ contains a 2-edge-cut $S$, where $e\in S$. Since $G$ is connected, $G/e$ is clearly connected. Moreover, by Lemma~\ref{lem1} (iii), $$F(G/e,\lambda)=F(G,\lambda)\not\equiv 0,$$ 
which implies that $G/e$ is bridgeless.
Also, $G/e$ has only real flow roots and $G/e$  contains the non-integral flow root $x$ with $|E(G/e)|<|E(G)|$, contradicting the minimality of $G$.
Hence Claim~\ref{cl3} holds.
\claimend

\Clm{cl5}{$G$ has no proper $3$-edge-cut.}

\proof
Assume that $G$ has a proper $3$-edge-cut $S$. Then the removal of $S$ separates $G$ into two subgraphs $H_1$ and $H_2$, where $|E(H_1)|,|E(H_2)|>0$. Also, by Claim~\ref{cl3}, $S$ is a minimal edge-cut, implying that $H_1,H_2$ are both connected.

Let $G_i$ be the graph obtained from $G$ by contracting $E(H_{3-i})$ for $i=1,2$. Then each $G_i$ is connected and contains fewer edges than $G$. Moreover, by Lemma~\ref{lem1} (iv), $$F(G_1,\lambda)F(G_2,\lambda)=F(G,\lambda)(\lambda-1)(\lambda-2)\not\equiv 0,$$
which implies that $G_1,G_2$ are bridgeless, $G_1$, $G_2$ have only real flow roots, and at least one of $G_1$, $G_2$  has the non-integral flow root $x$,  contradicting the minimality of $G$.
Hence Claim~\ref{cl5} holds.
\claimend

\Clm{cl6}{$n\ge 4$.}

\proof
Note that the minimum degree of $G$ is at least three as $G$ is 3-edge-connected.

Suppose $n=2$. 
Then there are at least three edges in $G$, implying that $m\ge 3$. 
Also, Lemma~\ref{lem:upper} indicates that $m\le 2n-1=3$. Hence $m=3=2n-1$. Again by Lemma~\ref{lem:upper}, $F(G,\lambda)=(\lambda-1)(\lambda-2)$, a contradiction to the assumption of $G$.

Suppose $n=3$. Then every vertex in $G$ has degree at least three, implying $2m\ge 3n=9$. Thus $m\ge 5$. Also, Lemma~\ref{lem:upper}  indicates that $m\le 2n-1=5$. Hence $m=5=2n-1$. Then by Lemma~\ref{lem:upper}, $F(G,\lambda)=(\lambda-1)(\lambda-2)^2$, also a contradiction to the assumption of $G$.

Then the claim follows from Claim~\ref{cl2}.
\claimend

\Clm{cl6+1}{$n\ge 2k+1$.}

\proof
By Claims~\ref{cl3},~\ref{cl5} and~\ref{cl6}, 
we have $n\ge 4$ and 
$G$ is a $3$-edge-connected 
graph which has no proper $3$-edge-cut.
Then Lemma~\ref{prop:dong} 
implies that $n\ge 2k+1$.
\claimend

\Clm{cl7}
{
$k=0$.
}

\proof
Suppose that $k\geq2$. Since $n\ge 2k+1$ by Claim~\ref{cl6+1}, we have $\frac{4(k-1)^2}{n-2k}\ge 0$. Then~\eqref{eq:dong} indicates
\[
    m\geq 2n+2k-3+\frac{4(k-1)^2}{n-2k}\geq 2n+2k-3>2n-1, %\ge n+1,
\]
a contradiction to Lemma~\ref{lem:upper}.

Suppose that $k=1$. Then \eqref{eq:dong} indicates
\[
m\ge 2n+2-3=2n-1. %. \ge n+1.
\]
By Lemma~\ref{lem:upper}, we have $m=2n-1$  and 
$F(G,\lambda)=(\lambda-1)(\lambda-2)^{n-1}$,
a contradiction  to the assumption of $G$.

Hence $k=0$.
\claimend

\Clm{cl8}
{
$n=4$.
}

\proof
By Claim~\ref{cl3}, every
vertex in $G$ has degree at least three, while by Claim~\ref{cl7}, every
vertex in $G$ has degree at most three.
Hence $G$ is cubic and $m=3n/2$. Then~\eqref{eq:dong} indicates
\[
    \frac{3n}{2}\geq2n-3+\frac{4}{n},
\]
or equivalently,
\[
    (n-2)(n-4)\leq0.
\]
Thus $n=4$ follows from  Claim~\ref{cl6}.
\claimend

\Clm{cl9}
{$G$ has only integral flow roots.}

\proof
Now $G$ is a cubic, $3$-edge-connected and loopless
multigraph on four vertices. We shall show that $G$ is exactly $K_4$. Equivalently, we shall prove that $G$ contains no parallel edges.

Suppose that $G$ has parallel edges between vertices $u,v$ in $G$. Since $G$ is cubic, there are at most two edges between $\{u,v\}$ and $V(G)\setminus \{u,v\}$ in $G$, which indicates an edge-cut of size at most two in $G$, a contradiction to the fact that $G$ is $3$-edge-connected. 

Consequently, $G$ is  $K_4$, where
\[
    F(K_4,\lambda)=(\lambda-1)(\lambda-2)(\lambda-3).
\]
Hence the claim holds.
\claimend

Claim~\ref{cl9} contradicts the assumption of $G$. This completes the proof.
\qed

\medskip
\noindent\textbf{Declaration on the use of AI:}
During the preparation of this work, the authors used AI systems to assist in exploring potential approaches to the proof of Theorem~\ref{thm:main}. All AI-generated suggestions were
verified and refined by the authors, who take full responsibility for the correctness and originality of the paper.

\noindent {\bf Data availability statement}:  Not applicable.

\section*{Acknowledgment}
This work is supported by the National Natural Science Foundation of China (No. 12501498), the Natural Science Foundation of Xiamen, China (No. 3502Z202571026), and the Foundation for Cultivated Young Talents of Fujian Province, China (No. 2025350064).

\end{document}